\documentclass[runningheads,a4paper,envcountsame]{llncs}
\usepackage[utf8]{inputenc} 
\usepackage[T1]{fontenc} 
\usepackage{lmodern}
\usepackage[english]{babel} 
\usepackage{amsmath,amssymb,amsfonts}

\let\doendproof\endproof
\renewcommand\endproof{~\hfill\qed\doendproof}

\usepackage{array}
\usepackage{enumerate} 
\usepackage{fancyhdr}
\usepackage{arcs}
\usepackage{tabvar}
\usepackage{mathrsfs}
\usepackage{dsfont}
\usepackage{stmaryrd}
\newcommand{\qedhere}{}
\usepackage{geometry}
\usepackage{tikz}
\usetikzlibrary{shapes,arrows,calc}

\renewcommand{\*}{\times}
\renewcommand{\geq}{\geqslant}
\renewcommand{\leq}{\leqslant}
\DeclareMathOperator{\mad}{\mathrm{mad}}
\DeclareMathOperator{\ad}{\mathrm{ad}}

\renewcommand{\choose}[2]{\left(\begin{smallmatrix} #1\\#2\end{smallmatrix}\right)}

\newcounter{config}

\newcounter{regle}

\spnewtheorem{obs}[theorem]{Observation}{\bfseries}{\rmfamily}

\title{Coloring squares of graphs with mad constraints}
\date{\today}
\author{Hervé Hocquard\inst{1} \and Seog-Jin Kim\inst{2} \and Théo Pierron\inst{1}}

\institute{Univ. Bordeaux, Bordeaux INP, CNRS, LaBRI, UMR5800, F-33400 Talence, France, \email{firstname.lastname@u-bordeaux.fr} \and 
Department of Mathematics Education, Konkuk University, Seoul, 05029, South Korea, \email{skim12@konkuk.ac.kr}} 

\begin{document}
\maketitle
\noindent
\makebox[\linewidth]{\small \today}
\vspace{0.5cm}
\begin{abstract}
  A proper vertex $k$-coloring of a graph $G=(V,E)$ is an assignment
  $c:V\to \{1,2,\ldots,k\}$ of colors to the vertices of the graph
  such that no two adjacent vertices are associated with the same
  color. The square $G^2$ of a graph $G$ is the graph defined by
  $V(G)=V(G^2)$ and $uv \in E(G^2)$ if and only if the distance
  between $u$ and $v$ is at most two. We denote by $\chi(G^2)$ the
  chromatic number of $G^2$, which is the least integer $k$ such that
  a $k$-coloring of $G^2$ exists. By definition, at least
  $\Delta(G)+1$ colors are needed for this goal, where $\Delta(G)$
  denotes the maximum degree of the graph $G$. In this paper, we prove
  that the square of every graph $G$ with $\mad(G)<4$ and
  $\Delta(G) \geqslant 8$ is $(3\Delta(G)+1)$-choosable and even
  correspondence-colorable. Furthermore, we show a family of
  $2$-degenerate graphs $G$ with $\mad(G)<4$, arbitrarily large
  maximum degree, and $\chi(G^2)\geqslant \frac{5\Delta(G)}{2}$,
  improving the result of Kim and Park~\cite{KP}.



\end{abstract}

This work was supported by the National Research Foundation of Korea
(NRF) grant funded by the Korea government (MSIT)
(NRF-2018R1A2B6003412).

\section{Introduction}
A proper vertex $k$-coloring of a graph $G=(V,E)$ is an assignment
$c:V\to \{1,2,\ldots,k\}$ of colors to the vertices of the graph such
that no two adjacent vertices are associated with the same color. The
square $G^2$ of a graph $G$ is the graph defined by $V(G)=V(G^2)$ and
$uv \in E(G^2)$ if and only if the distance between $u$ and $v$ is at
most two. We denote by $\chi(G^2)$ the chromatic number of $G^2$,
which is the least integer $k$ such that a $k$-coloring of $G^2$
exists. In other words, it is a stronger variant of graph coloring
where every two vertices within distance two have to receive different
colors. By definition, at least $\Delta(G)+1$ colors are needed for
this goal, where $\Delta(G)$ denotes the maximum degree of the graph
$G$. Indeed, if we consider a vertex of maximal degree and its
neighbors, they form a set of $\Delta(G)+1$ vertices, any two of which
are adjacent or have a common neighbor. Hence, at least $\Delta(G)+1$
colors are needed to color properly $G^2$. This subject was initiated
by Kramer and Kramer in~\cite{KK} and was intensively studied
afterwards especially for planar graphs. In 1977, Wegner
proposed~\cite{wegner} the following conjecture.

\begin{conjecture}[\cite{wegner}]
  \label{conj:wegner}
 If G is a planar graph, then:
 \begin{itemize}
 \item[$\bullet$] $\chi(G^2) \le 7$ if $\Delta(G)=3$
 \item[$\bullet$] $\chi(G^2) \le \Delta(G)+5$ if $4 \le \Delta(G) \le 7$
  \item[$\bullet$] $\chi(G^2) \le \lfloor\frac{3\Delta(G)}{2}\rfloor + 1$ if $\Delta(G) \ge 8$.
 \end{itemize}
\end{conjecture}

Let
$\mad(G)=\max\left\{\frac{2|E(H)|}{|V(H)|},\;H \subseteq G\right\}$ be
the \emph{maximum average degree} of a graph $G$, where $V(H)$ and
$E(H)$ are the sets of vertices and edges of $H$, respectively.  This
is a conventional measure of sparseness of an arbitrary graph (not
necessary planar). For more details on this invariant see {\it e.g.}
\cite{Coh10,Toft}.

Hosseini, Dolama and Sopena in~\cite{HDS} first made the link between
the maximum average degree and the chromatic number of the square of a
graph.  They proved the following result.

\begin{theorem}[\cite{HDS}]
  \label{thm:hds}
  Let $G$ be a graph with $\mad(G) < \frac{16}{7}$. Then, $\chi(G^2)=\Delta(G)+1$.
\end{theorem}

Recently, following problem was considered in \cite{cc} and has received some attentions.  
\begin{problem}[\cite{cc}] \label{Q-main}
For each integer $k \geq 2$, what is  $\max \{\chi(G^2)\mid  \mad(G) < 2k\}$? 
\end{problem}

For $k = 2$, Charpentier~\cite{cc} conjectured that
$\chi(G^2)\leq 2\Delta(G)$ if $\mad(G) < 4$, but it was disproved
in~\cite{KP} by constructing a graph $G$ such that
$\chi(G^2) = 2\Delta(G)+2$ and $\mad(G)<4$. Charpentier~\cite{cc}
proved that for sufficiently large $\Delta(G)$,
$\chi(G^2) \leq 3 \Delta(G) + 3$ if $\mad(G) < 4$.  Thus the results
in~\cite{cc} and~\cite{KP} implies that
\begin{equation} \label{bounds-original}
2 \Delta(G) +2 \leq \max \{\chi(G^2) \mid \mad(G) < 4\} \leq 3 \Delta(G) + 3.
\end{equation}


In this paper, we study Problem~\ref{Q-main} and we show that there
exists a family of graphs $G$ with $\mad(G) < 4$ and arbitrarily large
maximum degree such that $\chi(G^2) \geq \frac{5\Delta(G)}{2}$
(Theorem~\ref{thm:lower}).  We also show that
$\chi(G^2) \leq 3\Delta(G) +1$ if $\mad(G) < 4$ and $\Delta(G) \geq 8$
(Theorem~\ref{thm:main}). Note that the upper bounds
$\chi(G^2) \leq 3\Delta(G) +1$ are tight for $\Delta(G) \leq 4$.
These results improve the bounds on (\ref{bounds-original}) to
\begin{equation} \label{bounds-new}
\frac{5\Delta(G)}{2}\leq \max \{\chi(G^2) \mid \mad(G) < 4\} \leq 3 \Delta(G) + 1.
\end{equation}

We also prove upper bounds of $\chi(G^2)$ for arbitrarily integer
$k \geq 3$ and $\mad(G) < 2k$.  Charpentier proved~\cite{cc} that
roughly $(2k-1)\Delta$ colors are sufficient to color the square of
every graph $G$ with $\mad(G)<2k$ and $\Delta(G)=\Delta$. For
completeness, we give a proof of this result in
Section~\ref{sec:ghost}. However, we use another method called {\em
  ghost discharging}, that we present in Section~\ref{sec:ghost}.

In Section~\ref{sec:mad4upper}, we give the proof of upper bounds of
$\chi(G^2)$ for $\mad(G) < 4$, and in Section~\ref{sec:mad4lower}, we
present a generic construction that allows to extend the lower bound
obtained in~\cite{KP} for graphs with $\mad <4$.



\section{Generic Upper Bound}
\label{sec:ghost}

In this section, we include a proof of the following result for
completeness. 
\begin{theorem}[\cite{cc}]
  \label{thm:cc}
  Let $k$ be an integer and $G$ be a graph with $\mad(G) <2k$. Then
  \[\chi(G^2)\leqslant \max \{(2k-1)\Delta(G)-k^2+k+1,
    (2k-2)\Delta(G)+2k^3+k^2+2,(k-1)\Delta(G)+k^4+2k^3+2\}\]
\end{theorem}

In the following, we give two improvements: first, we rewrite the
original proof using only degeneracy. This allows to directly extend
Theorem~\ref{thm:cc} to generalized notions of coloring such as
list-coloring, or correspondence coloring~\cite{DP}.  Moreover, the
original proof uses discharging. We give a shorter proof using a
variant of discharging relying on the notion of \emph{ghost vertices}
defined below. This allows to fix some errors and inaccuracies of the
original proof. We actually prove the following.

\begin{theorem}
  \label{thm:maincc}
  Let $k$ be an integer and $G$ be a graph with $\mad(G) <2k$. Then
  $G^2$ is $f(k,\Delta)$-degenerate, where
  $f(k,\Delta)= \max \{(2k-1)\Delta(G)-k^2+k,
  (2k-2)\Delta(G)+2k^3+k^2+1, (k-1)\Delta(G)+k^4+2k^3+1\}$. 
\end{theorem}

To prove this result, we use the discharging method. This method was
introduced in~\cite{wernicke} to study the Four Color Conjecture. It
has been used to prove many results on sparse graphs (for example
planar, or with bounded mad), culminating with the Four Color Theorem
from~\cite{AH1,AH2}. This method leads to two-step proofs. In a first
step, we prove that if $G$ is a minimum counterexample to the theorem,
it cannot contain some patterns. Then, we prove that every graph from
a given class should contain at least one of these patterns. Put
together, these assertions prove that every graph from the given class
satisfies the theorem.

We thus assume that the theorem is false and take a graph $G$ with
$\mad(G)<2k$ and maximum degree $\Delta$, such that $G^2$ is not
$f(k,\Delta)$-degenerate. In subsection~\ref{sub:cc_config}, we give
some configurations and show they are not contained in $G$ (such a
configuration is called reducible). Then, in
Subsection~\ref{sub:cc_rules}, we use the ghost vertices method to
reach a contradiction.

\subsection{Reducible configurations}
\label{sub:cc_config}

Given a vertex $v\in V(G)$, we denote by $d(v)$ its degree in $G$, and
by $D(v)$ the number of $(k+1)^+$-vertices adjacent to $v$ in $G$. 

\begin{proposition}
  \label{prop:C1}
  The graph $G$ does not contain a $k^-$-vertex $u$ adjacent to a
  vertex $v$ with $D(v)\leqslant k$.
\end{proposition}

\begin{proof}
  Assume that $G$ contains such a configuration. By minimality,
  $(G\setminus uv)^2$ is $f(k,\Delta)$-degenerate. Take $\sigma$ an
  ordering witnessing this degeneracy, and remove $u,v$ and every
  $k^-$-vertex of $G$ from $\sigma$.

  We prove that $v$ has at most $f(k,\Delta)$ neighbors in $G^2$ that
  remains in $\sigma$. Then, since each $k^-$-vertex is adjacent to at
  most $k\Delta<f(k,\Delta)$ vertices in $G^2$, we obtain that
  $G^2$ is $f(k,\Delta)$-degenerate, a contradiction.

  By hypothesis, $D(v)\leqslant k$. Thus, the number of vertices
  appearing before $v$ in $\sigma$ is at most
  \[D\Delta+(\Delta-D)(k-1)\leqslant
    k\Delta+(\Delta-k)(k-1)=(2k-1)\Delta-k^2+k\leqslant f(k,\Delta)\qedhere\]
\end{proof}

\begin{proposition}
  \label{prop:C2}
  The graph $G$ does not contain a $k^-$-vertex $u$ with a neighbor
  $v$ satisfying:
  \begin{itemize}
  \item[$\bullet$] $k <D(v)< 2k$
  \item[$\bullet$] $v$ has at most $k-1$ neighbors $w$ with
    $D(w)\geqslant \frac{2k^2}{D(v)-k}$.
  \end{itemize}
\end{proposition}

\begin{proof}
  Assume that $G$ contains such a configuration.  Again, consider an
  ordering $\sigma$ witnessing that $(G\setminus uv)^2$ is
  $f(k,\Delta)$-degenerate, and remove $u,v$ and every $k^-$-vertex of
  $G$ from $\sigma$. Denote by $h$ the number of neighbors $w$ of $v$
  satisfying $D(w)\geqslant \frac{2k^2}{D(v)-k}$. By hypothesis,
  $h<k$.

  Again, since a $k^-$-vertex has at most $k\Delta$ neighbors in $G^2$
  and $k\Delta \leqslant f(k,\Delta)$, it is sufficient to prove that
  $v$ has at most $f(k,\Delta)$ neighbors in $G^2$ that remain in
  $\sigma$. The number of such vertices is at most
  \[h\Delta+(D(v)-h)\frac{2k^2}{D(v)-k} +
    (\Delta-D(v))(k-1)=(k+h-1)\Delta-D(v)(k-1)+2k^2+\frac{2k^2(k-h)}{D(v)-k}\]
  Since $h<k$, this is a decreasing function of $D(v)$. Hence it is at most
  \[(k+h-1)\Delta+k^2+1+2k^2(k-h)\]
  \begin{itemize}
  \item[$\bullet$] If $\Delta\geqslant 2k^2$, this is increasing in $h$, and thus
    at most
    \[(2k-2)\Delta+3k^2+1\leqslant f(k,\Delta)\]
  \item[$\bullet$] Otherwise, it is decreasing in $h$, thus at most
    \[(k-1)\Delta+2k^3+k^2+1<f(k,\Delta)\qedhere\]
  \end{itemize}
\end{proof}

To state the last reducible configuration, we introduce the notion of
light vertex. If $k <D<2k$, a vertex $v$ is \emph{$D$-light} if
\begin{itemize}
\item[$\bullet$] either $k+1\leqslant D(v) < k+\frac{Dk}{2D-2k}$ and $v$ has at
  most $k-1$ neighbors $w$ with
  $D(w)\geqslant \frac{k^2D}{(D-k)(D(v)-k)}$.
\item[$\bullet$] or $k+\frac{Dk}{2D-2k} \leqslant D(v) <\frac{Dk}{D-k}$ and $v$
  has less than $D(v)-\frac{(D(v)-2k)D}{2k-D}$ neighbors $w$ with
  $D(w) \geqslant 2k$.
\end{itemize}
We may then state our last reducible configuration.

\begin{proposition}
  \label{prop:C3}
  The graph $G$ does not contain a vertex $u$ with $k<D(u)<2k$, no
  $k^-$-neighbor and adjacent to a $D(u)$-light vertex $v$.
\end{proposition}

\begin{proof}
  Assume that $G$ contains such a configuration. Again, consider an
  ordering $\sigma$ witnessing that $(G\setminus uv)^2$ is
  $f(k,\Delta)$-degenerate, and remove $u,v$ and every $k^-$-neighbor
  of $v$ from $\sigma$. We consider the ordering $\sigma'$ obtaining
  by appending $v$, then $u$, then the removed $k^-$-vertices to
  $\sigma$.

  Again, since a $k^-$-vertex has at most $k\Delta$ neighbors in $G^2$
  and $k\Delta \leqslant f(k,\Delta)$, it is sufficient to prove that
  $u$ and $v$ have at most $f(k,\Delta)$ neighbors in $G^2$ that
  appear previously in $\sigma'$.

  We first count the $(k+1)^+$-neighbors of $u$ in $G^2$: there are
  $v$, the $(k+1)^+$-neighbors of $v$, and the neighbors of the
  $D(u)-1$ neighbors of $u$. Thus, there are at most
  \[1+D(v)+(D(u)-1)\Delta\leqslant
    1+\frac{D(u)k}{D(u)-k}+(2k-2)\Delta\] neighbors of $u$. This is a
  decreasing function of $D(u)$, hence it is at most
  \[(2k-2)\Delta+k^2+k+1 \leqslant f(k,\Delta)\]

  For $v$, we consider two cases according to the definition of
  $D(u)$-light vertex.
  \begin{itemize}
  \item[$\bullet$] Assume that $k+1\leqslant D(v) < k+\frac{D(u)k}{2D(u)-2k}$ and
    $v$ has $h$ neighbors $w$ with
    $D(w)\geqslant \frac{k^2D(u)}{(D(u)-k)(D(v)-k)}$.
    
    Then, in $G^2$, the number of $(k+1)^+$-neighbors $v$ besides $u$
    is at most:
    \begin{align*}
      &(\Delta-D(v))(k-1)+ h\Delta+(D(v)-h)\frac{k^2D(u)}{(D(u)-k)(D(v)-k)}\\
      &=(k+h-1)\Delta-D(v)(k-1)+\frac{k^2D(u)}{D(u)-k}+\frac{(k-h)k^2D(u)}{(D(u)-k)(D(v)-k)}.
    \end{align*}
    Since $h<k$, this is a decreasing function of $D(v)$, hence at
    most
    \[(k+h-1)\Delta-(k+1)(k-1)+(k-h+1)k^2+\frac{(k-h+1)k^3}{D(u)-k}\]
    This is decreasing in $D(u)$, hence at most 
    \[(k+h-1)\Delta-(k+1)(k-1)+(k-h+1)(k^3+k^2)\]
    \begin{itemize}
    \item If $\Delta\geqslant k^3+k^2$, this is an increasing function
      of $h$, hence it is at most
      \[(2k-2)\Delta+k^2+1+2k^3\leqslant f(k,\Delta)\]
    \item Otherwise, this is a decreasing function of $h$, hence it is
      at most
      \[(k-1)\Delta+k^4+2k^3+1\leqslant f(k,\Delta)\]
    \end{itemize}
  \item[$\bullet$] Assume that
    $k+\frac{D(u)k}{2D(u)-2k} \leqslant D(v) <\frac{D(u)k}{D(u)-k}$
    and $v$ has $h$ neighbors $w$ with $D(w) \geqslant 2k$, where $h$
    is less than $D(v)-\frac{(D(v)-2k)D(u)}{2k-D(u)}$.

    First observe that
    \[D(v)-\frac{(D(v)-2k)D(u)}{2k-D(u)}=\frac{2D(u)k-(2D(u)-2k)D(v)}{2k-D(u)}\] which is a
    decreasing function of $D(v)$, hence it is at most $k$ since
    $D(v)\geqslant k+\frac{D(u)k}{2D(u)-2k}$. Hence $h\leqslant k-1$.

    Consider the $(k+1)^+$-neighbors of $v$ in $G^2$ (excepted
    $u$). There are at most
    \[h\Delta +(\Delta-D(v))(k-1)+ (2k-1)(D(v)-h)=(k+h-1)\Delta+kD(v)-h(2k-1)\]
    such vertices.  This is increasing in $D(v)$, hence at most
    \[(k+h-1)\Delta+\frac{k^2D(u)}{D(u)-k}-k-h(2k-1)\] 
    This is at decreasing in $D(u)$, hence at most
    \[(k+h-1)\Delta+k^2(k+1)-k-h(2k-1)\] 
    \begin{itemize}
    \item If $\Delta \geqslant 2k-1$, this is increasing in $h$, hence at most
      \[(2k-2)\Delta+k^3-k^2-4k-1\leqslant f(k,\Delta)\]
    \item Otherwise, this is decreasing in $h$, hence at most 
      \[(k-1)\Delta+k^3+k^2-k\leqslant f(k,\Delta)\]
    \end{itemize}
  \end{itemize}
\end{proof}

\subsection{Ghost vertices}
\label{sub:cc_rules} 
To reach a contradiction, we use the discharging method. Moreover, we
consider a so called \emph{Ghost vertices method}, introduced earlier by Bonamy, Bousquet and Hocquard~\cite{BonamyEtAl}.

We begin by giving a weight $\omega(v)=d(v)-2k$ to each vertex of
$G$. We then design some rules in order to redistribute the weights on
$G$ so that the final weights $\omega'$ satisfy: 
\begin{itemize}
\item[$\bullet$] $\omega'(v)\geqslant 0$ if $d(v)>k$.
\item[$\bullet$] $\omega'(v)\geqslant d(v)+D(v)-2k$ if $d(v)\leqslant k$.
\end{itemize}
In this case, we say that $v$ is \emph{happy}. We first prove that we
reach a contradiction if every vertex is happy. Let $H$ be the
subgraph of $G$ induced by the $(k+1)^+$-vertices. Observe that
\[\sum_{u\in G\setminus H} D(u) = |E(H,G\setminus H)|=\sum_{u\in H} (d(u)-D(u))\]
Thus, we have
\begin{align*}
  \sum_{u\in H} (D(u)-2k)&= \sum_{u\in H} (d(u)-2k)- \sum_{u\in H} (d(u)-D(u))\\
&= \sum_{u\in G} (d(u)-2k)- \sum_{u\in G\setminus H} (d(u)-2k)- \sum_{u\in H} (d(u)-D(u))\\
&= \sum_{u\in G} \omega'(u) - \sum_{u\in G\setminus H} (d(u)-2k)- \sum_{u\in H} (d(u)-D(u))\\
&=\sum_{u\in H} \omega'(u) +\sum_{u\in G\setminus H} (\omega'(u)-d(u)+2k)- \sum_{u\in H} (d(u)-D(u))\\
&=\sum_{u\in H} \omega'(u) +\sum_{u\in G\setminus H} (\omega'(u)-d(u)+2k-D(u))\\
\end{align*}
Each term of the two last sums is non-negative, hence we obtain that
$\mad(G)\geqslant \ad(H) \geqslant 2k$, a contradiction. This thus
ends the proof of Theorem~\ref{thm:maincc}.


We consider three discharging rules that we apply in order:
\begin{itemize}
\item[$\bullet$] $R_0$: Every vertex in $H$ gives $1$ to each of its neighbors
  outside $H$.
\item[$\bullet$] $R_1$: Every vertex $u$ with $D(u)\geqslant 2k+1$ gives
  equitably all its weight to its neighbors $v$ in $H$ with $D(v)<2k$.
\item[$\bullet$] $R_2$: Every vertex with positive weight gives equitably all its
  weight to its neighbors in $H$ with negative weight.
\end{itemize}

We now prove that every vertex is happy. First note that due to $R_0$,
every vertex $v$ in $G\setminus H$ receives a weight of $D(v)$, and is
not affected by $R_1$ and $R_2$. Its final weight is then at least
$d(v)-2k+D(v)$, hence it is happy.

We may thus only consider vertices in $H$. Let $u$ be such a
vertex. We separate several cases depending on $D(u)$. Observe that
after $R_0$, $u$ has weight $D(u)-2k$. We now prove that $u$ ends up
with non-negative weight after $R_1$ and $R_2$. Observe that if, after
applying $R_0$ or both $R_0,R_1$ a vertex ends with non-negative
weight, then it still has non-negative weight after applying the
remaining rules.

\begin{itemize}
\item[$\bullet$] Assume that $D(u)\leqslant k$. Then since $u\in H$, we have
  $d(u)\geqslant k+1$, so $u$ has a $k^-$-neighbor in $G$. This is
  impossible by Proposition~\ref{prop:C1}.
\item[$\bullet$] Assume that $D(u)\geqslant 2k$. Then $u$ has positive weight
  after $R_0$ and $u$ is happy.
\item[$\bullet$] Assume that $k < D(u) < 2k$ and $u$ has a $k^-$-neighbor in
  $G$. Then by Proposition~\ref{prop:C2}, $u$ has at least $k$
  neighbors $v$ with $D(v)\geqslant \frac{2k^2}{D(u)-k}$.
  
  Observe that since $D(u)<2k$, we have $D(v)> 2k$, hence $w$ gives
  weight to $u$ by $R_1$. The amount of such weight is at least
  \[\frac{D(v)-2k}{D(v)}=1-\frac{2k}{D(v)}\geqslant 2-\frac{D(u)}{k}\]
  since the middle term is increasing in $D(v)$. Since there are at
  least $k$ such vertices $w$, $u$ receives at least $2k-D(u)$ and
  thus ends up with non-negative weight after $R_1$. Therefore, $u$ is
  happy.
\item[$\bullet$] Finally, assume that $k < D(u) < 2k$ and $u$ has no
  $k^-$-neighbor in $G$. Let $v$ be a neighbor of $u$ in $H$. We prove
  that $v$ gives at least $\frac{2k}{D(u)}-1$ to $u$ by $R_1$ or
  $R_2$. If true, this would imply that $u$ receives at least
  $2k-D(u)$ and thus ends up with non-negative weight. We separate
  several cases:
  \begin{itemize}
  \item Assume that $D(v)\geqslant \frac{D(u)k}{D(u)-k}$. Then since
    $D(u)<2k$, we have $D(v)>2k$, hence $v$ gives weight to $u$ by
    $R_1$. The amount given is at least
    \[\frac{D(v)-2k}{D(v)}=1-\frac{2k}{D(v)}\geqslant
      1-\frac{2k(D(u)-k)}{D(u)k}=\frac{2k}{D(u)}-1\] as requested.
  \item Assume that
    $k+\frac{D(u)k}{2D(u)-2k} \leqslant D(v)
    <\frac{D(u)k}{D(u)-k}$. Then, by Proposition~\ref{prop:C3}, $v$
    has at least $D(v)-\frac{(D(v)-2k)D(u)}{2k-D(u)}$ neighbors $w$
    with $D(w) \geqslant 2k$.

    Observe that $D(v)\geqslant 2k$, hence $v$ gives weight to $u$ by
    $R_1$. Note that $v$ does not give any weight to neighbors $w$
    with $D(w)\geqslant 2k$, hence $v$ distributes its weight among at
    most $\frac{(D(v)-2k)D(u)}{2k-D(u)}$ vertices. Thus $u$ receives at least
    \[\frac{(D(v)-2k)(2k-D(u))}{(D(v)-2k)D(u)}=\frac{2k}{D(u)}-1\]
  \item Assume that $k+1\leqslant D(v) < k+\frac{D(u)k}{2D(u)-2k}$.
    Then by Proposition~\ref{prop:C3}, $v$ has at least $k$ neighbors
    $w$ with $D(w)\geqslant \frac{k^2D(u)}{(D(u)-k)(D(v)-k)}$. Observe
    that in this case, $D(w)\geqslant 2k+1$ and $D(v) <2k$, hence $w$
    gives weight to $v$ by $R_1$. The transfered amount is at least
    \[\frac{D(w)-2k}{D(w)}=1-\frac{2k}{D(w)}\geqslant 1-\frac{2(D(u)-k)(D(v)-k)}{kD(u)}\]
    Thus, the weight of $v$ after $R_1$ is at least
    \[D(v)-2k+k\left(1-\frac{2(D(u)-k)(D(v)-k)}{kD(u)}\right)=(D(v)-k)\left(\frac{2k}{D(u)}-1\right)\]
    This is non-negative, hence either $u$ has non-negative weight
    after $R_1$, or it receives weight from $v$ by $R_2$. In this
    case, observe that $v$ has at least $k$ neighbors with
    non-negative charge, hence the transfered weight is at least
    \[\frac{D(v)-k}{D(v)-k}\left(\frac{2k}{D(u)}-1\right)=\frac{2k}{D(u)}-1\]
  \end{itemize}
  Therefore, $u$ ends up happy, and we obtain the required
  contradiction. This ends the proof of Theorem~\ref{thm:maincc}.
\end{itemize}

\section{Upper bound when $\mad<4$}
\label{sec:mad4upper}

In this section, we prove the following result.
\begin{theorem}
  \label{thm:main}
  Let $G$ be a graph with $\mad(G)<4$ and $\Delta \geqslant 8$. Then
  $\chi(G^2)\leqslant 3\Delta(G)+1$.
\end{theorem}

Observe that this improves Theorem~\ref{thm:cc} when
$8\leqslant\Delta \leqslant 21$. To prove Theorem~\ref{thm:main}, we
actually prove that, for every $\Delta\geqslant 8$, if $G$ is a graph
with $\mad(G)<4$ and $\Delta(G)\leqslant \Delta$, then $G^2$ is
$3 \Delta$-degenerate. This implies Theorem~\ref{thm:main}, as well as its
generalizations for list and correspondence coloring.

By contradiction, take a graph $G$ with $\mad(G)<4$ and
$\Delta(G) \geqslant \Delta$, and assume that $G^2$ is not
$3\Delta$-degenerate. Moreover, assume that $G$ has minimum number of edges
among all the graphs having this property. We say that an ordering of
the vertices of $G$ is \emph{good} if every vertex is appears after at
most $3\Delta$ of its neighbors in $G^2$.

We again use the discharging method. In Subsection~\ref{sub:configs},
we prove that $G$ does not contain some configurations. Then, in
Subsection~\ref{sub:weights}, we obtain a contradiction using some
weight transfer argument.

\subsection{Reducible configurations}
\label{sub:configs}

To introduce the configurations, we need some terminology.
\begin{definition}
  Let $v$ be a $d$-vertex of $G$, with $d_i$ neighbors of degree $i$
  ($i=2,3$). If $d\geqslant 4$, we say that:
  \begin{itemize}
  \item[$\bullet$] $v$ is \emph{nice} if $d-d_2\geqslant 8$.
  \item[$\bullet$] $v$ is \emph{good} if $d-d_2\geqslant 6$.
  \item[$\bullet$] $v$ is \emph{weakly good} if $d-d_2=5$.
  \item[$\bullet$] $v$ is \emph{weakly bad of type $1$} if $d-d_2=4$ and $d_3=0$,
    and \emph{weakly bad of type $2$} if $d-d_2=4$ and $d_3=1$.
  \item[$\bullet$] $v$ is \emph{bad} if $d-d_2=3$.
  \end{itemize}
\end{definition}

According to this definition, we may first prove the following
classification of the vertices of $G$.
\begin{proposition}
  \label{prop:types}
  Every $4^+$-vertex of $G$ is bad, weakly bad, weakly good or good.
\end{proposition}

\begin{proof}
  Assume there is a $4^+$-vertex $v$ of $G$ which is not bad, weakly
  bad, weakly good nor good. This implies that either
  $d(v)-d_2(v)\leqslant 2$ or $d(v)-d_2(v)=4$ and $d_3(v)\geqslant 2$.

  In the first case, since $d(v)\geqslant 4$, $v$ has a $2$-neighbor
  $w$. By minimality, take $\sigma$ a good ordering for
  $(G\setminus vw)^2$. Let $\sigma'$ be the ordering obtained by
  removing $v$ and its $2$-neighbors from $\sigma$, and adding them
  (in this order) at the end of $\sigma$. We show that $\sigma'$ is a
  good ordering.

  Note that $v$ has at most $2\Delta+\Delta-2=3\Delta-2$ neighbors
  appearing before it in $\sigma'$. Its $2$-neighbors are preceded by
  at most $2\Delta$ neighbors in $\sigma'$. Thus $\sigma'$ is a good
  ordering for $G$.

  In the second case, let $w_1,w_2$ be two $3$-neighbors of $v$. By
  minimality, take a good ordering $\sigma$ of $(G\setminus
  vw_1)^2$. Let $\sigma'$ be obtained by removing $v,w_1,w_2$ and the
  $2$-neighbors of $v$ from $\sigma$ and adding them at the end of
  $\sigma$. Note that $v$ appears after $2\Delta+\Delta-4+4=3\Delta$
  of its neighbors. Similarly, $w_1,w_2$ appear after $2\Delta+4$ of
  their neighbors. Finally, the $2$-neighbors of $v$ have at most
  $2\Delta$ neighbors in $G^2$, hence previously in $\sigma'$. The
  ordering $\sigma'$ is then good for $G$, a contradiction.
\end{proof}

We may now introduce the reducible configurations we consider. We
roughly show that vertices with small $d-d_2$ are not close in $G$. We
study the neighborhood of the vertices of each type, beginning with
the $3^-$-vertices.
\begin{proposition}
  \label{prop:small}
  In $G$, no $3^-$-vertex is adjacent to a $3^-$-vertex.
\end{proposition}

\begin{proof}
  Let $u,v$ be adjacent $3^-$-vertices of $G$. By minimality, let
  $\sigma$ be a good ordering for $(G\setminus uv)^2$. Remove $u$ and
  $v$ from $\sigma$ and add them at the end of $\sigma$. In the
  obtained coloring $\sigma'$, both $u$ and $v$ are preceded by at
  most $2\Delta+2$ neighbors. Since $\Delta > 2$, $\sigma'$ is a good
  ordering for $G^2$, a contradiction.
\end{proof}

\begin{proposition}
  \label{prop:bad}
  In $G$, every $4^+$-neighbor from a bad vertex is not bad.
\end{proposition}

\begin{proof}
  Let $u,v$ be adjacent bad vertices of $G$. Let $w$ be a $2$-neighbor
  of $v$. By minimality, take a good ordering $\sigma$ of
  $(G\setminus vw)^2$. We remove $v$ and the $2$-neighbors of $u$ and
  $v$ from $\sigma$ and add them in this order at the end of
  $\sigma$. In the obtained coloring $\sigma'$, the vertex $v$ appears
  after at most $3\Delta$ of its neighbors. Moreover, each of the (at
  most) $2\Delta-6$ uncolored $2$-vertices has at most $\Delta+4$
  neighbors in $\sigma$, hence appears after at most $3\Delta-2$
  neighbors in $\sigma'$. Hence $\sigma'$ is a good ordering for
  $G^2$, a contradiction.
\end{proof}

\begin{proposition}
  \label{prop:weakbad}
  Let $v$ be a bad neighbor in $G$ from a weakly bad vertex $u$. Then
  $v$ has at least two nice neighbors.
\end{proposition}

\begin{proof}
  Assume that $v$ has a neighbor $w$ such that $w$ is not nice and $w \neq u$.
  Since $v$ is bad, it has a neighbor $x$ of degree $2$. By
  minimality, we take a good ordering $\sigma$ of $(G\setminus
  vx)^2$. We remove $v$ and the $2$-vertices incident to $v,w$ from
  $\sigma$ and add them in this order at the end of $\sigma$.

  In the obtained ordering $\sigma'$, the vertex $v$ has at most
  $2\Delta+1+d(w)-d_2(w)$ neighbors before it. Since $w$ is not nice,
  this is bounded by $2\Delta+8$ and by $3\Delta$ since
  $\Delta \geqslant 8$. Moreover, each $2$-vertex has at most
  $2\Delta$ neighbors, hence $\sigma'$ is a good ordering for $G'$, a
  contradiction.
\end{proof}

\begin{proposition}
  \label{prop:weakbad2}
  In $G$, each weakly bad vertex of type 2 has at least one good
  neighbor.
\end{proposition}

\begin{proof}
  Let $u$ be a weakly bad vertex of type $2$ without nice
  neighbor. Let $v_1,v_2,v_3$ be the neighbors of $u$ that are not
  good and let $w$ be the $3$-neighbor of $u$. By minimality, take a
  good ordering $\sigma$ of $(G\setminus uw)^2$. We define an ordering
  $\sigma'$ by removing $u,w$ and the $2$-vertices adjacent to
  $u,v_1,v_2,v_3$ from $\sigma$ and adding them in this order at the
  end of $\sigma$.

  The number of neighbors of $u$ preceding it in $\sigma'$ is at most
  $\Delta-2+d(v_1)-d_2(v_1)+d(v_2)-d_2(v_2)+d(v_3)-d_2(v_3)\leqslant
  \Delta+13$. Since $\Delta \geqslant 8$, this is bounded by $3\Delta$.
  
  The vertex $w$ has degree $3$, hence has at most $3\Delta$ neighbors
  in $G^2$. Finally, the remaining $2$-vertices have at most $2\Delta$
  neighbors. Therefore, $\sigma'$ is a good ordering for $G^2$, a
  contradiction.
\end{proof}

\begin{proposition}
  \label{prop:weakgood}
  In $G$, each weakly good vertex has at most three neighbors that are
  $3$-vertices or bad vertices with at most one nice neighbor.
\end{proposition}

\begin{proof}
  Let $u$ be a weakly good vertex of $G$ with at least four neighbors
  $v_1,\ldots,v_4$ that have degree $3$ or are bad vertices with at
  most one nice neighbor. 

  If $v_1$ has degree $3$, we take a good ordering $\sigma$ of
  $(G\setminus uv_1)^2$ by minimality. Otherwise, $v_1$ is a bad
  vertex so it has a $2$-neighbor $w$. In this case, we take $\sigma$
  as a good ordering of $(G\setminus v_1w)^2$.

  In both cases, we denote by $\sigma'$ the ordering obtained by
  removing $u, v_1,\ldots,v_4$ and their $2$-neighbors from $\sigma$.

  To construct a good ordering for $G^2$, we first consider the bad
  vertices among $v_1,\ldots,v_4$. Assume that $v_i$ is bad for some
  $i=1,\ldots,4$ and denote by $x$ one of its non-nice neighbors. We
  remove the $2$-neighbors of $x$ from $\sigma'$ and add $v_i$ at the
  end of $\sigma'$. Note that $v_i$ has at most
  $2\Delta+1+d(x)-d_2(x)\leqslant 2\Delta + 8$ appearing in $\sigma'$,
  which is less than $3\Delta$ since $\Delta \geqslant 8$.

  We then add $u$ at the end of $\sigma'$. It is still a good ordering
  since $u$ has at most $2\Delta+7\leqslant 3\Delta$ neighbors in
  $\sigma'$. We then add the remaining vertices $v_i$ (of degree $3$)
  to the end of $\sigma'$. Note that they have at most $2\Delta+5$
  neighbors in $\sigma'$.

  Finally, we add all the remaining $2$-vertices at the end of
  $\sigma'$. Then $\sigma'$ is a good coloring for $G^2$, a
  contradiction.
\end{proof}

\subsection{Discharging part}
\label{sub:weights}

We may now reach a contradiction. We give an initial weight
$\omega(v)=d(v)-4$ to each vertex $v$ of $G$. Since $\mad(G)< 4$, the
total weight is negative. 

Observe that the ghost method we use in Section~\ref{sec:ghost} seems
not to be useful there. Indeed, we could have used $2^-$-vertices as
ghosts. In this case, we should have designed discharging rules such
that the following assertions hold:
\begin{itemize}
\item[$\bullet$] If $v$ is a $3^+$-vertex, then $v$ ends up with non-negative weight.
\item[$\bullet$] If $v$ is a $2$-vertex, then $v$ ends up with weight at least $d(v)-4+d_{3^+}(v)$.
\end{itemize}
Since $2^-$-vertices are not adjacent by Proposition~\ref{prop:small},
the last constraint can be rewritten as: $2^-$-vertices have to end
with non-negative weight. Thus, we basically end up with what we
actually have to prove. We now introduce some discharging rules. 

We first apply the following rule: each vertex gives $1$ to its
neighbors of degree $2$ and $\frac{1}{3}$ to its neighbors of degree
$3$. Observe that nice vertices are all good. We may then state our
other rules:
\begin{enumerate}
\item Every nice vertex gives $\frac{1}{2}$ to its bad neighbors.
\item Every $4^+$-vertex which is not nice gives $\frac{1}{3}$ to each
  bad neighbor having at most one nice neighbor.
\item Every good vertex gives $\frac{1}{3}$ to its weakly bad
  neighbors of type 2.
\end{enumerate}

We now show that every vertex of $G$ ends up with non-negative weight,
which is a contradiction with the hypothesis $\mad(G)<4$. We separate
several cases according to the type of vertices we consider.
\medskip
\paragraph{\bf $3^-$-vertices.}

By the first rule, each $2$-vertex $v$ of $G$ receives $1$ from each
of its neighbors. Moreover, $v$ does not lose any weight, thus its
final weight is $\omega'(v)=2-4+2\* 1=0$.

Similarly, each $3$-vertex ends up with non-negative weight since it
does not lose weight and each of its neighbors gives it $\frac{1}{3}$
by the first rule. So $\omega'(v)=3-4+3\*\frac{1}{3}=0$.

\medskip
\paragraph{\bf Bad vertices.}

Let $v$ be a bad vertex of $G$. After applying the first rule, $v$ has
weight $-1$. Recall that bad vertices are not good, and no neighbor of
$v$ is bad by Proposition~\ref{prop:bad}, so $v$ does not lose some
additional weight.

Due to Rule 1, if $v$ has at least two nice neighbors, then $v$ ends
up with $\omega'(v)=-1+2\* \frac{1}{2}=0$. Otherwise, Rule 2 applies,
and $v$ receives $3\*\frac{1}{3}$ from its $4^+$-neighbors. Thus
$\omega'(v)\geqslant 0$.

\medskip
\paragraph{\bf Weakly bad vertices.}

Let $v$ be a weakly bad vertex of $G$.  Recall that $v$ is not
good. Moreover, if $v$ has a bad neighbor $w$, then
Proposition~\ref{prop:weakbad} ensures that $w$ has two nice
neighbors, so $v$ does not lose any weight during the second phase.

Thus, if $v$ has type $1$, then it ends up with no weight after the
first phase so its final weight is $\omega'(v)=0$.

Otherwise, $v$ has type $2$, so it has weight $-\frac{1}{3}$ after the
first phase. By Proposition~\ref{prop:weakbad2}, it has a good
neighbor, so it receives $\frac{1}{3}$ by Rule 3, and ends up with
weight $0$.

\medskip
\paragraph{\bf Weakly good vertices.}

Let $v$ be a weakly good vertex of $G$. After giving weight to
$2$-vertices, $v$ ends up with weight $1$. Note that $v$ is not good,
so $v$ only loses weight for each vertex of degree $3$ or to bad
neighbors with at most one nice neighbor. By
Proposition~\ref{prop:weakgood}, $v$ has at most three such neighbors,
so $v$ ends up with non-negative weight.

\medskip
\paragraph{\bf Good vertices.}

Let $v$ be a good vertex of $G$ of degree $d$ with $d_2$ neighbors of
degree $2$. If $v$ is not nice, it loses $\frac{1}{3}$ for at most
$d-d_2$ neighbors, hence its final weight is at most
$d-4-d_2-\frac{d-d_2}{3}=\frac{2}{3}(d-d_2)-4\geqslant 0$ since
$d-d_2\geqslant 6$.

Otherwise, $v$ loses $\frac{1}{2}$ for at most $d-d_2$ neighbors, so
its final weight is at most
$d-4-d_2-\frac{d-d_2}{2}=\frac{d-d_2}{2}-4\geqslant 0$ since
$d-d_2\geqslant 8$.

By Proposition~\ref{prop:types}, every vertex has been considered by
one of the previous arguments. Therefore, every vertex ends up with
non-negative weight, which concludes.

\section{Lower Bound}
\label{sec:mad4lower}

In this section, we investigate the lower bounds for $\chi(G^2)$ when
$G$ is a graph with $\mad(G)<4$. We first consider graphs with small
$\Delta$, here $\Delta\leqslant 5$.

\subsection{Small $\Delta$}
For $\Delta=1$, $G$ is a matching, hence $G^2$ is $2$-colorable, which
is tight when $G=P_2$. 

For $\Delta=2$, $G$ is a path or a cycle, hence $G^2$ is
$4$-degenerated and $5$-colorable. This is tight, as shown by
$C_5$. 

For $\Delta=3$, the Petersen graph needs $10$ colors since it has
diameter two. This achieves the upper bound $3\Delta+1$ for
$\Delta=3$.
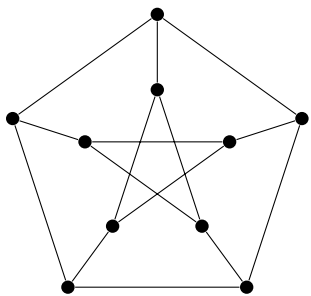
\begin{figure}[!ht]
\centering
\begin{tikzpicture}[v/.style={fill=black,minimum size =5pt,ellipse,inner sep=1pt},node distance=1.5cm]
\node[v] (v1) at (90:2){};
\node[v] (v2) at (162:2){};
\node[v] (v3) at (234:2){};
\node[v] (v4) at (306:2){};
\node[v] (v5) at (18:2){};

\node[v] (w1) at (90:1){};
\node[v] (w2) at (162:1){};
\node[v] (w3) at (234:1){};
\node[v] (w4) at (306:1){};
\node[v] (w5) at (18:1){};

\draw (v1) -- (v2) -- (v3) -- (v4) -- (v5) -- (v1);
\draw (w1) -- (w3) -- (w5) -- (w2) -- (w4) -- (w1);
\draw (v1) -- (w1);
\draw (v2) -- (w2);
\draw (v3) -- (w3);
\draw (v4) -- (w4);
\draw (v5) -- (w5);
\end{tikzpicture}
\caption{$\chi(G^2)=10$, $\mad<4$, $\Delta=3$}
\end{figure}

For $\Delta=4$, the following graph also has diameter two and
thus needs $13$ colors, also achieving the bound $3\Delta+1$. 

\begin{figure}[!ht]
\centering
\begin{tikzpicture}[v/.style={fill=black,minimum size =5pt,ellipse,inner sep=1pt},node distance=1.5cm]
\node[v] (v1) at (0:2){};
\node[v] (v2) at (45:2){};
\node[v] (v3) at (90:2){};
\node[v] (v4) at (135:2){};
\node[v] (v5) at (180:2){};
\node[v] (v6) at (225:2){};
\node[v] (v7) at (270:2){};
\node[v] (v8) at (315:2){};
\node[v] (v0) at (0,0){};

\draw (v1) to node[v,midway] (v14) {} (v4);
\draw (v2) to node[v,midway] (v27) {} (v7);
\draw (v3) to node[v,midway] (v36) {} (v6);
\draw (v5) to node[v,midway] (v58) {} (v8);

\draw (v27) --(v0) -- (v14);
\draw (v36) -- (v0) -- (v58);
\draw (v1) -- (v2) -- (v3) -- (v4) -- (v5) -- (v6) -- (v7) -- (v8) -- (v1);
\draw (v1) to (v6);
\draw (v2) to (v5);
\draw (v3) to (v8);
\draw (v4) to (v7);
\end{tikzpicture}
\caption{$\chi(G^2)=13$, $\mad<4$, $\Delta=4$}
\end{figure}
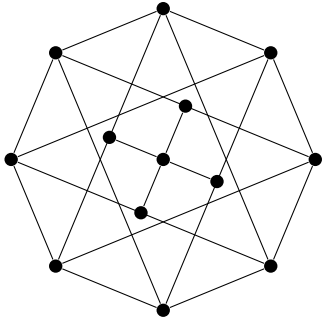

Finally, for $\Delta=5$, the following graph needs $15$ colors (the
black and red vertices induce a clique in the square). This graph is
build from a Petersen graph adding five vertices of degree 3 linked by
paths of length 2. Note that this graph has mad 4. However, removing
the red part leads to a graph of mad less than $4$ that needs $14$
colors.

\begin{figure}[!ht]
\centering
\begin{tikzpicture}[v/.style={fill=black,minimum size =5pt,ellipse,inner sep=1pt},node distance=1.5cm]
\node[v] (v1) at (90:1){};
\node[v] (v2) at (162:1){};
\node[v] (v3) at (234:1){};
\node[v] (v4) at (306:1){};
\node[v] (v5) at (18:1){};

\node[v] (w1) at (90:2){};
\node[v] (w2) at (162:2){};
\node[v] (w3) at (234:2){};
\node[v] (w4) at (306:2){};
\node[v] (w5) at (18:2){};

\node[v,red] (x1) at (54:1.5){};
\node[v] (x2) at (126:1.5){};
\node[v] (x3) at (198:1.5){};
\node[v] (x4) at (270:1.5){};
\node[v] (x5) at (342:1.5){};

\draw (v1) -- (v2) -- (v3) -- (v4) -- (v5) -- (v1);
\draw[bend right=90] (w1) to (w3) to (w5) to (w2) to (w4) to (w1);
\draw[bend right=10] (x4) to (v1);
\draw[bend right=10] (x5) to (v2);
\draw[bend right=10,red] (x1) to (v3);
\draw[bend right=10] (x2) to (v4);
\draw[bend right=10] (x3) to (v5);

\draw (v1) -- (w1);
\draw (v2) -- (w2);
\draw (v3) -- (w3);
\draw (v4) -- (w4);
\draw (v5) -- (w5);

\draw  (w1) -- (x2) -- (w2) -- (x3) -- (w3) -- (x4) -- (w4) -- (x5) -- (w5);
\draw[red] (w5) -- (x1) -- (w1);
\draw[gray] (x3) -- node[v,midway,gray] {}(x5) -- node[v,midway,gray]{}(x2) -- node[v,midway,gray]{}(x4);
\draw[red] (x3) -- node[v,midway,gray] {} (x1) -- node[v,midway,gray] {} (x4);
\end{tikzpicture}
\caption{$\chi(G^2)=14$, $\mad<4$, $\Delta=5$}
\end{figure}
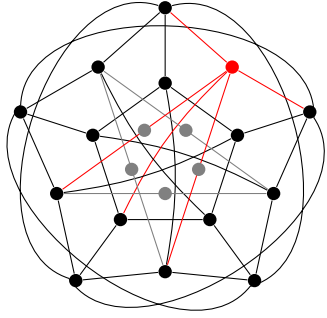

\subsection{Large $\Delta$}

We now give a construction improving the result of~\cite{KP} when
$\mad(G)<4$, even when $G$ is $2$-degenerate. We actually prove the
following result.

\begin{theorem}
  \label{thm:lower}
  There exists a family of $2$-degenerate graphs $G$ with $\mad(G)<4$,
  arbitrarily large maximum degree, and
  $\chi(G^2)\geqslant \frac{5\Delta(G)}{2}$.
\end{theorem}

Let $t$ be an integer. We define $G_t$ as the graph obtained from
$K_5$ by applying successively the two following operations:
\begin{itemize}
\item[$\bullet$] Replacing each edge $e$ by a copy of $K_{2,t}$ by identifying
  the endpoints of the edge with the two vertices in the same
  partition. We denote by $V_e$ the $t$ vertices added while replacing
  $e$.
\item[$\bullet$] For each pair of non-incident edges $e,f$, we add a path over
  two edges between each pair of vertices in $V_e\* V_f$.
\end{itemize}

For $t>2$, observe that $\Delta(G_t)=4t$ and $G_t$ is $2$-degenerated
(consider the vertices by reversing their order of creation). Thus
$\mad(G_t)<4$.

Moreover, the vertices in $\cup_{e\in E(K_5)} V_e$ induce a clique of
size $10t$ in $G_t^2$. Therefore, we have
$\chi(G_t^2)\geqslant 10t=\frac{5\Delta(G_t)}{2}$.

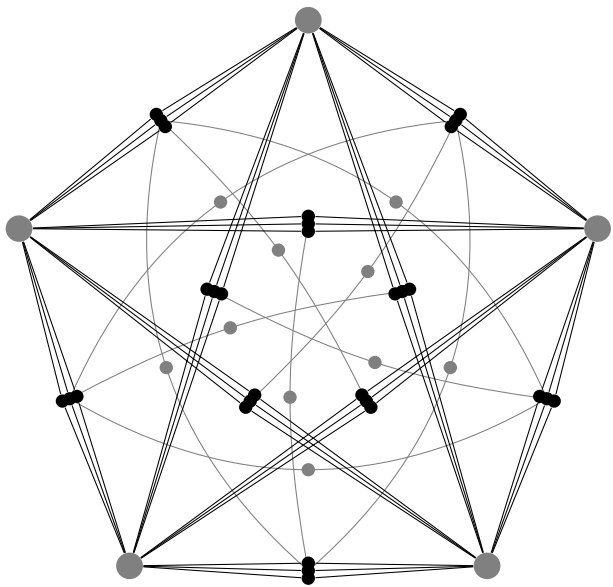
\begin{figure}[!ht]
\centering
\begin{tikzpicture}[v/.style={fill=gray,minimum size =10pt,ellipse,inner sep=1pt},node distance=1.5cm,v2/.style={fill=black,minimum size =5pt,ellipse,inner sep=1pt},node distance=1.5cm]
\node[v] (v1) at (90:4){};
\node[v] (v2) at (162:4){};
\node[v] (v3) at (234:4){};
\node[v] (v4) at (306:4){};
\node[v] (v5) at (18:4){};
\node[v2] (vv11) at (126:3.3){};
\node[v2] (vv21) at (198:3.3){};
\node[v2] (vv31) at (270:3.3){};
\node[v2] (vv41) at (342:3.3){};
\node[v2] (vv51) at (54:3.3){};

\draw[bend right,gray] (vv11) to node[v2,midway,gray] {} (vv31);
\draw[bend left,gray] (vv11) to node[v2,midway,gray] {} (vv41);
\draw[bend left,gray] (vv21) to node[v2,midway,gray] {} (vv51);
\draw[bend right,gray] (vv21) to node[v2,midway,gray] {} (vv41);
\draw[bend right,gray] (vv31) to node[v2,midway,gray] {} (vv51);

\node[v2] (vv12) at (126:3.4){};
\node[v2] (vv22) at (198:3.4){};
\node[v2] (vv32) at (270:3.4){};
\node[v2] (vv42) at (342:3.4){};
\node[v2] (vv52) at (54:3.4){};

\node[v2] (vv13b) at (162:1.3){};
\node[v2] (vv35b) at (306:1.3){};
\node[v2] (vv52b) at (90:1.3){};
\node[v2] (vv24b) at (234:1.3){};
\node[v2] (vv41b) at (18:1.3){};

\draw[bend left=10,gray] (vv11) to node[v2,midway,gray] {} (vv35b);
\draw[bend left=10,gray] (vv21) to node[v2,midway,gray] {} (vv41b);
\draw[bend left=10,gray] (vv31) to node[v2,midway,gray] {} (vv52b);
\draw[bend left=10,gray] (vv41) to node[v2,midway,gray] {} (vv13b);
\draw[bend left=10,gray] (vv51) to node[v2,midway,gray] {} (vv24b);

\node[v2] (vv1) at (126:3.2){};
\node[v2] (vv2) at (198:3.2){};
\node[v2] (vv3) at (270:3.2){};
\node[v2] (vv4) at (342:3.2){};
\node[v2] (vv5) at (54:3.2){};

\node[v2] (vv13c) at (162:1.4){};
\node[v2] (vv35c) at (306:1.4){};
\node[v2] (vv52c) at (90:1.4){};
\node[v2] (vv24c) at (234:1.4){};
\node[v2] (vv41c) at (18:1.4){};

\node[v2] (vv13a) at (162:1.2){};
\node[v2] (vv35a) at (306:1.2){};
\node[v2] (vv52a) at (90:1.2){};
\node[v2] (vv24a) at (234:1.2){};
\node[v2] (vv41a) at (18:1.2){};

\draw (v1) -- (vv1) -- (v2) -- (vv2) -- (v3) -- (vv3) -- (v4) -- (vv4) -- (v5) -- (vv5) -- (v1);
\draw (v1) -- (vv11) -- (v2) -- (vv21) -- (v3) -- (vv31) -- (v4) -- (vv41) -- (v5) -- (vv51) -- (v1);
\draw (v1) -- (vv12) -- (v2) -- (vv22) -- (v3) -- (vv32) -- (v4) -- (vv42) -- (v5) -- (vv52) -- (v1);
\draw (v1) -- (vv13a) -- (v3) -- (vv35a) -- (v5) -- (vv52a) -- (v2) -- (vv24a) -- (v4) -- (vv41a) -- (v1);
\draw (v1) -- (vv13b) -- (v3) -- (vv35b) -- (v5) -- (vv52b) -- (v2) -- (vv24b) -- (v4) -- (vv41b) -- (v1);
\draw (v1) -- (vv13c) -- (v3) -- (vv35c) -- (v5) -- (vv52c) -- (v2) -- (vv24c) -- (v4) -- (vv41c) -- (v1);





\end{tikzpicture}
\caption{The graph $G_t$, black vertices induce a clique in $G_t^2$}
\end{figure}

Observe that a similar construction can be done starting from any
cliques $K_n$. For $n=6$, this gives the same lower bound. However,
when $n\geqslant 7$, the clique number of $G_t^2$ is
$\frac{tn(n-1)}{2}$ while $\Delta(G_t)=t(\frac{n(n-1)}{2}- 2n+3)$, which
gives a worse lower bound.

\section{Conclusion}

In this paper we investigate lower and upper bounds for square
coloring of graphs with maximum average degree bounded, especially
with $\mad < 4$. Reducing the gap between the lower bounds and the
upper bounds in (\ref{bounds-new}) is an interesting problem.  So we
have the following question.

\begin{question}
  Is there integer $D$ such that every graph $G$ with
  $\Delta(G) \geq D$ and $\mad(G) < 4$ has
  $\chi(G^2) \leq \frac{5\Delta(G)}{2}$?
\end{question}

Note that the constructions in Theorem \ref{thm:lower} are actually
2-degenerate. So we propose the following question.

\begin{question}
  Is there integer $D$ such that every graph $G$ with
  $\Delta(G) \geq D$ has $\chi(G^2) \leq \frac{5 \Delta(G)}{2}$ if $G$
  is 2-degenerate?
\end{question}

Moreover, while this lower bound cannot be strengthened using larger
cliques, there may be a way of generalizing the given
construction. Indeed, instead of considering a clique and replacing
edges by a bipartite graph $K_{2,p}$, consider an hypergraph on $kr$
vertices where all the hyperedges of size $k$ are present, and replace
each hyperedge by a bipartite graph $K_{k,p}$ (the construction for
Theorem~\ref{thm:lower} is the case $k=2$). Denote by $V_e$ the
vertices added while applying this construction to the hyperedge $e$
and by $G$ the obtained graph. The problem is then to add paths of
length 2 between $V_e$ and $V_f$ for every pair $(e,f)$ of non
incident hyperedges. Given a set of $k$ pairwise non-incident edges
$\{e_1,\ldots,e_k\}$, we can add $p^2$ vertices of degree $k$ to $G$
such that $V_{e_1}\cup\cdots \cup V_{e_k}$ induces a clique in
$G^2$. However, if this is done for every set of $k$ pairwise
non-incident edges, the degree of vertices in each $V_e$ is too large
to obtain a good bound.

Thus, we need to find a suitable packing of the hyperedges of the
considered hypergraph. In other terms, we have to solve the following
problem:

\begin{question}
  Given an integer $k$, is there an integer $r$ and set $\mathcal{S}$
  such that the following holds?
  \begin{enumerate}
  \item Each element of $\mathcal{S}$ is a set of $k$ pairwise
    disjoint $k$-subsets of $\llbracket 1,rk \rrbracket$.
  \item If $S,T$ are two $k$-subsets of $\llbracket 1,rk\rrbracket$,
    there exists an element of $\mathcal{S}$ containing both $S$ and
    $T$.
  \item If $S$ is a $k$-subset of $\llbracket 1,rk\rrbracket$, $S$ is
    contained in at most $ \frac{1}{k-1}\choose{k(r-1)}{k}$ elements
    of $\mathcal{S}$.
  \end{enumerate}
\end{question}

Solving this problem with $r=k$ would yield a bound of the same order
than in~\cite{KP}. However, we believe that the parameter $r$ can be
optimized (as done in Section~\ref{sec:mad4lower}, with $k=2$ and
$r=3$) to obtain much better values. Note that for our purposes, the
bound of Item 3 can be weakened up to an additive constant, or even to
$\frac{1}{k-1}\choose{k(r-1)}{k}(1+o_r(1))$ (with possibly some
consequences on the resulting lower bound).



\begin{thebibliography}{99}
\bibitem{cc}
C. Charpentier,
2-distance coloring of not-so-sparse graphs, Personal communication.
\end{thebibliography}
\end{document}